\documentclass[a4paper,10pt]{amsart}

\usepackage{amssymb}
\usepackage{fullpage}
\usepackage[OT4]{fontenc}

\theoremstyle{definition}
\newtheorem{definition}{Definition}
\theoremstyle{plain}
\newtheorem{theorem}[definition]{Theorem}
\newtheorem{proposition}[definition]{Proposition}
\newtheorem{lemma}[definition]{Lemma}
\newtheorem{conjecture}[definition]{Conjecture}

\theoremstyle{remark}
\newtheorem{remark}[definition]{Remark}

\DeclareMathOperator{\Ass}{Ass}

\DeclareMathOperator{\reg}{reg}

\def\field{\mathbb{K}}
\def\PP{\mathbb{P}}
\def\ZZ{\mathbb{Z}}
\def\seqmul{\overline{m}}

\let\to\longrightarrow

\let\un\underline

\begin{document}

\title{Containments of symbolic powers of ideals of generic points in $\PP^3$}

\author{Marcin Dumnicki}

\dedicatory{
Institute of Mathematics, Jagiellonian University, \\
ul. \L{}ojasiewicza 6, 30-348 Krak\'ow, Poland \\
Email address: Marcin.Dumnicki@im.uj.edu.pl\\
}



\begin{abstract}
We show that the Conjecture of Harbourne and Huneke, $I^{(Nr-(N-1))} \subset M^{(r-1)(N-1)}I^{r}$
holds for ideals of generic (simple) points in $\PP^3$. As a result, for such ideals we prove the following
bounds, which can be recognized as generalizations of Chudnovsky bounds: $\alpha(I^{(3m-k)}) \geq m\alpha(I)+2m-k$,
for any $m \geq 1$ and $k=0,1,2$. Moreover, we obtain lower bounds for the Waldshmidt
constant for such ideals.
\end{abstract}

\maketitle

\section{Introduction}

Let $\field$ be a field of chracteristic zero, let $\field[\PP^N]=\field[x_0,\dots,x_N]$ denote the ring of coordinates of
the projective space with standard grading. Let $I \subset \field[\PP^N]$ be a homogeneous ideal.
By $m$-th symbolic power we define
$$
I^{(m)} = \field[\PP^N] \cap \big( \bigcap_{\mathfrak{p} \in \Ass(I)} I^{m} \field[\PP^N]_{\mathfrak{p}} \big),
$$
where the intersection is taken in the ring of fractions of $\field[\PP^N]$.
By a \emph{fat points ideal} we denote the ideal
$$I = \bigcap_{j=1}^{n} \mathfrak{m}_{p_j}^{m_j},$$
where $\mathfrak{m}_{p}$ denotes the ideal of forms vanishing at a point $p \in \PP^N$, and $m_j \geq 1$ are integers.
Observe that for a fat points ideal $I$ as above
$$I^{(m)} = \bigcap_{j=1}^{n} \mathfrak{m}_{p_j}^{m_j m}.$$
Let $M=(x_0,\dots,x_N) \subset \field[\PP^N]$ be the maximal homogeneous ideal.

A sequence $\seqmul=(m_1,\dots,m_n)$ of $n$ integers will be called a \emph{sequence of multiplicities}.
Define the \emph{ideal of generic fat points} in $\field[\PP^N]$ to be
$$I(\seqmul) = I(m_1,\dots,m_n) = \bigcap_{j=1}^{n} \mathfrak{m}_{p_j}^{m_j}$$
for points $p_1,\dots,p_n$ in general position in $\PP^N$ (for $m_j < 0$ we take $\mathfrak{m}_{p_j}^{m_j} = \field[\PP^N]$).
We will use the following notation:
$$m^{\times s} = \underbrace{(m,\dots,m)}_{s}.$$
If all multiplicities are equal to one, we say that $I$ is a \emph{simple points ideal}.

In \cite{Laz} it is shown that $I^{(Nr)} \subset I^{r}$ for any radical homogeneous ideal $I \subset \field[\PP^N]$.
This result has been improved in a sequence of papers by \cite{BoHa1}, \cite{primer}, \cite{HaHu} and the author to the following:

\begin{theorem}
Let $I$ be an ideal of a finite number of points in general position in $\PP^N$. Then
\begin{itemize}
\item
$I^{(3)} \subset I^2$ for $N=2$ (\cite{BoHa1}),
\item
$I^{(2r-1)} \subset I^r$ for $N=2$ (\cite{primer}),
\item
$I^{(2r)} \subset M^{r}I^{r}$ for $N=2$ (\cite{HaHu}),
\item
$I^{(3r)} \subset M^{2r}I^{r}$ for $N=3$ (\cite{Dum}),
\item
$I^{(2r-1)} \subset M^{r-1}I^{r}$ for $N=2$ (\cite{HaHu}),
\item
$I^{(Nr-(N-1))} \subset I^{r}$, where the number of points is sufficiently big (depending on $N$) (\cite{primer}).
\end{itemize}
\end{theorem}

It is also conjectured that

\begin{conjecture}
\label{mainconj}
Let $I$ be a simple points ideal in $\PP^N$. Then
\begin{itemize}
\item
$I^{(Nr-(N-1))} \subset I^{r}$ (Conjecture 4.1 in \cite{HaHu}),
\item
$I^{(Nr-(N-1))} \subset M^{(r-1)(N-1)}I^{r}$ (Conjecture 4.5 in \cite{HaHu})
\end{itemize}
for all $r \geq 1$.
\end{conjecture}

In the paper we show that Conjecture \ref{mainconj} holds for any number of generic points in $\PP^3$.

\begin{theorem}
\label{thm1}
Let $I$ be an ideal of simple points in general position in $\PP^3$. Then $I^{(3r-2)} \subset M^{2r-2}I^r$ for any $r \geq 1$, thus also
$I^{(3r-2)} \subset I^r$ for any $r \geq 1$. Additionally, we have $I^{(3r-1)} \subset M^{2r-1}I^r$ for any $r \geq 1$, which,
together with the main result in \cite{Dum} ($I^{(3r)} \subset M^{2r}I^r$) completes the picture. 
\end{theorem}

\begin{proof}
The proof of the theorem will be divided into six separated cases, depending on the number of points $n$ and $r$. Proposition
\ref{case1} deals with $n \geq 512$, $r \geq 3$, Proposition \ref{case2} with $5 \leq n \leq 511$, $r \geq 3$, Proposition \ref{case3}
with $n\geq 65$, $r=2$, Proposition \ref{case4} with $7 \leq n \leq 64$, $r=2$, Lemma \ref{singular} with $n=5,6$, $r=2$, and
Proposition \ref{case5} with $n \leq 4$ and arbitrary $r$. The case $r=1$ is trivial.
\end{proof}

For a homogeneous ideal $I \subset \field[\PP^N]$ define
$$I_t = \{ f \in I : \deg(f) = t \},$$
and
$$\alpha(I) = \min \{ t \geq 0 : I_t \neq 0 \}.$$

\begin{remark}
Chudnovsky \cite{Ch} conjectured that for an ideal of points in $\PP^3$ the following bound holds:
$$\alpha(I^{(r)}) \geq r \frac{\alpha(I)+2}{3}.$$
This bound follows from the containment $I^{(3r)} \subset M^{2r}I^r$ (see \cite{HaHu}), which has been proved in \cite{Dum} for
ideals of simple points in general position. Moreover, this containment obviously gives
\begin{equation}
\label{knownbound}
\alpha(I^{(3r)}) \geq r\alpha(I)+2r.
\end{equation}
Observe that the containment $I^{(3r-k)} \subset M^{2r-k}I^r$ gives the following bound for $\alpha(I^{(3r-k)})$, which can be treated as the generalization
of the Chudnovsky Conjecture and fits perfectly with \eqref{knownbound}:
$$\alpha(I^{(3r-k)}) \geq r\alpha(I)+2r-k.$$
Observe also that the containment $I^{(3r-3)} \subset M^{2r-3}I^r$ does not hold for $r=1$ (in fact we must define $M^{-1}$ in this case), so we focus on
the case $I^{(3r-k)} \subset M^{2r-k}I^r$ for $k=0,1,2$ and $r \geq 1$. There are also examples when the containment $I^{(2r)} \subset I^r$ does not hold
(see Corollary 1.1.1 in \cite{BoHa1}).
\end{remark}

We need to introduce the Waldschmidt constant of an ideal (cf. \cite{HaHu}).

\begin{proposition}
\label{Wc}
Let $I \subset K[\PP^{N}]$ be a non-zero homogeneous ideal. Define the \emph{Waldschmidt constant}
$$\gamma(I) = \lim_{m \to \infty} \frac{\alpha(I^{(m)})}{m}.$$
Then the following holds:
\begin{enumerate}
\item
$\gamma(I)$ exists and satisfies $\alpha(I^{(m)}) \geq m\gamma(I)$ for all $m \geq 1$,
\item
if $I \subset J$ are ideals of finite number of points then $\gamma(I) \geq \gamma(J)$,
\item
If $I$ is an ideal of fat points then $\gamma(I^{(r)}) = r\gamma(I)$.
\end{enumerate}
\end{proposition}

\begin{proof}
See \cite[Proposition 7]{Dum} and \cite[Lemma 2.3.1]{BoHa1}.
\end{proof}

Define $\gamma(N,n) = \gamma(I)$, where $I$ is the ideal of $n$ simple points in general position in $\PP^N$. The conjectural value for
$\gamma(N,n)$ is $\sqrt[N]{n}$ for $n$ big enough. Anyway, we always have the inequality $\gamma(N,n) \leq \sqrt[N]{n}$ (see \cite{dhst}).
Therefore, for $n$ big, we are interested in lower bounds for $\gamma(N,n)$. Since we deal with the case $N=3$ we will write
$\gamma(n)=\gamma(3,n)$. As a by-product we will show the following
$$\gamma(n) \geq 0.7787 \sqrt[3]{n} + 0.6142, \qquad \text{ for } n \geq 512.$$

\section{Bounds for Waldschmit constant}

\begin{proposition}
\label{cremona}
Let $\seqmul$ be a sequence of multiplicities. If 
$$I(m_1,m_2,m_3,m_4,\seqmul)_{t} \neq 0$$
then
$$I(m_1+k,m_2+k,m_3+k,m_4+k,\seqmul)_{t+k} \neq 0$$
for $k=2t-m_1-m_2-m_3-m_4$.
\end{proposition}

\begin{proof}
The proof involves using standard birational transformation for $\PP^3$, and some combinatorial arguments
to deal with the case of negative multiplicities. The full proof can be found in \cite{Dum} (see Proposition 8 there).
\end{proof}

The above operation on ideal and degree will be called ``Cremona operation''. Since we can permute multiplicities, Cremona operation can be performed
on any four of them. In such situations, we will indicate on which multiplicities (or on which points) it is performed.

\begin{proposition}
\label{cre8}
Let $m_1,\dots,m_r$, $t$ be integers.
If
$$ 4t > m_1+\ldots+m_8$$
then
$I(m_1,\dots,m_r)_{t}$ can be transformed by a sequence of Cremona operations into an ideal $I(m_1',\dots,m_r')_{t'}$ with arbitrarily large $t'$.
Moreover, these Cremona operations can be chosen to work only on $m_1,\dots,m_8$.
\end{proposition}

\begin{proof}
Let $I(0) = I(m_1,\dots,m_r)_{t}$, define inductively
$I(n)$ to be the ideal obtained from $I(n-1)$ by performing Cremona on multiplicities number 1, 2, 3, 4 and again
Cremona on multiplicities number 5, 6, 7, 8. Let $T(n)$ be the degree of $I(n)$, let $S_1(n)$ be the sum of first four
multiplicities in $I(n)$, let $S_2(n)$ be the sum of the next four multiplicities. Denote also
$s_1 = m_1+m_2+m_3+m_4$, $s_2=m_5+m+6+m_7+m_8$.
Obviously $T(0)=t$, $S_1(0)=s_1$, $S_2(0)=s_2$. Perform Cremona on $I(n-1)$ on first four
multiplicities. New degree is equal to $3T(n-1)-S_1(n-1)$, and new sum of first four multiplicities is equal to
$8T(n-1)-3S_1(n-1)$. Performing another Cremona on other four multiplicities gives recurrence relations for
$T(n)$, $S_1(n)$ and $S_2(n)$:
\begin{align}
\label{recrel1}
T(n) & = 9T(n-1)-3S_1(n-1)-S_2(n-1) \\
\label{recrel2}
S_1(n) & = 8T(n-1)-3S_1(n-1) \\
\label{recrel3}
S_2(n) & = 24T(n-1) - 8S_1(n-1)-3S_2(n-1).
\end{align}
We claim that
\begin{align*}
T(n) & = (8n^2+1)t - (2n^2+n)s_1 - (2n^2-n)s_2 \\
S_1(n) & = (16n^2-8n)t - (4n^2-1)s_1 - (4n^2-4n)s_2 \\
S_2(n) & = (16n^2+8n)t - (4n^2+4n)s_1 - (4n^2-1)s_2.
\end{align*}
Indeed, it is easy to observe that these polynomials evaluated at 0 give $t$, $s_1$, $s_2$ resp., and (using
any computer algebra system) that they satisfy recurrence relations \eqref{recrel1}, \eqref{recrel2} and \eqref{recrel3}.
Since
$$T(n)=(8t-2s_1-2s_2)n^2 + \text{ lower terms}$$
and $8t-2s_1-2s_2 > 0$ by assumption we get that $T(n)$ can take arbitrarily large values for $n \gg 0$.
\end{proof}

\begin{proposition}
\label{3crseq}
Let $I(M_1(n),\dots,M_{10}(n))_{T(n)}$ be the ideal obtained from the ideal $I(0^{\times 10})_{1}$ by $n$ iterations
of the following operation: Cremona based on points number 1,2,3,4, then 1,5,6,7, then 1,8,9,10. Then
$\lim T(n) = \infty$ for $n \to \infty$.
\end{proposition}

\begin{proof}
For simplicity we change the notation and assume that $I_{0}=I(0,0^{\times 3},0^{\times 3},0^{\times 3})_{1}$, while
$I(n)$ comes from $I(n-1)$ by six Cremona operations --- based on points number 1,2,3,4, then 1,5,6,7, then 1,8,9,10,
then (again) 1,2,3,4, then 1,5,6,7, then 1,8,9,10.
We claim that
$$I(n) = I(54n^2,(18n^2-6n)^{\times 3},(18n^2)^{\times 3},(18n^2+6n)^{\times 3})_{54n^2+1}.$$
The straightforward inductive proof of this fact is left to the reader.
\end{proof}

\begin{proposition}
\label{glue}
Let $m_1,\dots,m_r,m_1',\dots,m_s'$, $t$, $k$ be integers.
If
$$I(m_1,\dots,m_r)_{k} = 0$$
and
$$I(m_1',\dots,m_s',k+1)_{t} = 0$$
then
$$I(m_1,\dots,m_r,m_1',\dots,m_s')_{t} = 0.$$
\end{proposition}

\begin{proof}
\cite[Theorem 10]{Dumalg}.
\end{proof}

\begin{proposition}
\label{vargamma}
Define $b(n)$ for the following values on $n$:
$$
\begin{array}{c|cccccccccccccccc}
n & 1 & 2 & 3 & 4 & 5 & 6 & 7 & 8 & 12 & 14 & 16 & 17 & 21 & 24 & 30 & 35 \\ \hline
b(n) & 1 & 1 & 1 & \frac{4}{3} & \frac{5}{3} & \frac{12}{7} & \frac{28}{15} & 2 & \frac{126}{57} & \frac{7}{3} & \frac{22}{9} & \frac{5}{2} & \frac{8}{3} & \frac{107}{39} & 3 & \frac{67}{21}.
\end{array}
$$
Then $\gamma(n) \geq b(n)$ whenever $b(n)$ is defined.
\end{proposition}

\begin{proof}
We will consider each case separately. We note here that for $n \leq 8$ we have $\gamma(n)=b(n)$. Moreover, for $n \leq 8$ this equality
can be derived from much more general procedure of computing $\alpha(I^{(m)})$ for an ideal $I$ of at most 8 general points in $\PP^3$ and any $m\geq 1$ (see \cite{Laf}). 
\begin{itemize}
\item
$n=1,2,3$ is clear.
\item
$n=4$. Let $I$ be the ideal of 4 points in general position. If $(I^{(3m)})_{4m-1} \neq 0$ then, by Proposition \ref{cremona},
$I((-m-2)^{\times 4})_{-2} = \field[\PP^3]_{-2} \neq 0$, which is false. Hence $\alpha(I^{(3m)}) \geq 4m$ gives
$\gamma(I) \geq \frac{4}{3}$ by Proposition \ref{Wc}.
\item
$n=5$. Similarly, let $I$ be the ideal of 5 points in general position. If $(I^{(3m)})_{5m-1} \neq 0$ then, by Proposition \ref{cremona},
$I(3m,(-m-2)^{\times 4})_{3m-2} \neq 0$, which is false. Hence $\alpha(I^{(3m)}) \geq 5m$ gives
$\gamma(I) \geq \frac{4}{3}$.
\item
$n=6$. Consider $I((7m)^{\times 6})_{12m-1}$. We make the following Cremona transformations
$$
\begin{array}{c|cccccc|c}
t & m_1 & m_2 & m_3 & m_4 & m_5 & m_6 & k \\ \hline
12m-1 & \un{7m} & \un{7m} & \un{7m} & \un{7m} & 7m & 7m & -4m-2 \\
8m-3 & 3m-2 & 3m-2 & \un{3m-2} & \un{3m-2} & \un{7m} & \un{7m} & -4m-2 \\
4m-5 & \un{3m-2} & \un{3m-2} & & & \un{3m-2} & \un{3m-2} & -4m-2 \\
-7   & & & & & & &
\end{array}
$$
to show that this ideal is the zero ideal.
\item
$n=7$. Consider $I((15m)^{\times 7})_{28m-1}$. Again we make Cremona transformation. To make the presentation better, we first
show what happens with coefficients standing by $m$'s, and then what happens with free coefficients. In other words, we make several Cremona on
$I(15^{\times 7})_{28}$ obtaining $I((-1)^{\times 7})_{0}$, and then the same sequence of Cremona on $I(0^{\times 7})_{1}$ obtaining
$I((8)^{\times 4},4,8,8)_{15}$. By linearity, we obtain that $I((15m)^{\times 7})_{28m-1}$ can be transformed into
$I((-m-8)^{\times 6},-m-4)_{-15}$, which is the zero ideal.
$$
\begin{array}{c|ccccccc|c}
t & m_1 & m_2 & m_3 & m_4 & m_5 & m_6 & m_7 & k \\ \hline
28 & \un{15} & \un{15} & \un{15} & \un{15} & 15 & 15 & 15 & -4 \\
24 & \un{11} & 11 & 11 & 11 & \un{15} & \un{15} & \un{15} & -8 \\
16 & 3 & \un{11} & \un{11} & \un{11} & \un{7} & 7 & 7 & -8 \\
8 & \un{3} & \un{3} & 3 & 3 & -1 & \un{7} & \un{7} & -4 \\
4 & -1 & -1 & \un{3} & \un{3} & -1 & \un{3} & \un{3} & -4 \\
0 & -1 & -1 & -1 & -1 & -1 & -1 & -1
\end{array}
$$
$$
\begin{array}{c|ccccccc|c}
t & m_1 & m_2 & m_3 & m_4 & m_5 & m_6 & m_7 & k \\ \hline
1 & \un{0} & \un{0} & \un{0} & \un{0} & 0 & 0 & 0 & +2 \\
3 & \un{2} & 2 & 2 & 2 & \un{0} & \un{0} & \un{0} & +4 \\
7 & 6 & \un{2} & \un{2} & \un{2} & \un{4} & 4 & 4 & +4 \\
11 & \un{6} & \un{6} & 6 & 6 & 4 & \un{4} & \un{4} & +2 \\
13 & 8 & 8 & \un{6} & \un{6} & 4 & \un{6} & \un{6} & +2 \\
15 & 8 & 8 & 8 & 8 & 4 & 8 & 8 \\
\end{array}
$$
\item
$n=8$. Consider $I(m^{\times 8})_{2m-1}$. Observe that each Cremona on $I(1^{\times 8})_{2}$ gives the same ideal. Hence we are only
interested in free part $I(0^{\times 8})_{-1}$. By Proposition \ref{cre8} the ideal $I(0^{\times 8})_{1}$ can be transformed into
an ideal with sufficiently large degree (greater than $m$), hence $I(m^{\times 8})_{2m-1}$ can be transformed into $I(\dots)_{<0}$, which
completes the proof of this case.
\item
$n=12$. Consider $I((57m)^{\times 12})_{126m-1}$. Since $\gamma(5) \geq \frac{5}{3}$ we have, by Proposition \ref{Wc},
$\alpha(I((57m)^{\times 5})) \geq 95m$, hence
$$I((57m)^{\times 5})_{95m-1} = 0.$$
By Proposition \ref{glue} it suffices to show that $I(95m,(57m)^{\times 7})_{126m-1} = 0$.
Consider the following sequence of Cremona:
$$
\begin{array}{c|cccccccc|c}
t & m_1 & m_2 & m_3 & m_4 & m_5 & m_6 & m_7 & m_8 & k \\ \hline
126 & \un{95} & \un{57} & \un{57} & \un{57} & 57 & 57 & 57 & 57 & -14 \\
112 & \un{81} & 43 & 43 & 43 & \un{57} & \un{57} & \un{57} & 57 & -28 \\
84 & \un{53} & \un{43} & \un{43} & 43 & 29 & 29 & 29 & \un{57} & -28 \\
56 & 25 & 15 & 15 & \un{43} & \un{29} & \un{29} & \un{29} & 29 & -18 \\
38 & \un{25} & \un{15} & 15 & \un{25} & 11 & 11 & 11 & \un{29} & -18 \\
20 & 7 & & \un{15} & 7 & \un{11} & \un{11} & \un{11} & 11 & -8 \\
12 & \un{7} & & \un{7} & \un{7} & 3 & 3 & 3 & \un{11} & -8 \\
4 & & & & & \un{3} & \un{3} & \un{3} & \un{3} & -4 \\
0 & & & & & & & &
\end{array}
$$
Now perform the same sequence on the free part $I(0^{\times 8})_{-1}$. Observe that $I(0^{\times 8})_{1}$ is non-zero,
so after any sequence of Cremona it will give a non-zero ideal, hence with positive degree (degree equal 0 is excluded, since
$\dim_{\field} I_{0} \leq 0$ but $\dim_{\field} I(0^{\times 8})_{1} > 0$; we must know that Cremona preserves also the dimension over $\field$, but it is
immediate by adding simple points). Therefore free part will have negative degree, which completes the proof.
\item
$n=14$. Consider $I((3m)^{\times 14})_{7m-1}$.
It suffices to show that $I(5m,(3m)^{\times 9})_{7m-1} = 0$.
Observe that Cremona based on multiplicities $5,3,3,3$ applied to $I(5,3^{\times 9})_{7}$ gives exactly the same ideal,
since $k=2\cdot 7-5-3 \cdot 3 = 0$.
Thus everything relies on the free part $I(0^{\times 10})_{-1}$, which can have arbitrarily low degree according to Proposition \ref{3crseq}.
\item
$n=16$. Consider $I((9m)^{\times 16})_{22m-1}$.
It suffices to show that $I(18m,(9m)^{\times 8})_{22m-1} = 0$.
Consider the following sequence of Cremona:
$$
\begin{array}{c|ccccccccc|c}
t & m_1 & m_2 & m_3 & m_4 & m_5 & m_6 & m_7 & m_8 & m_9 & k \\ \hline
22 & \un{18} & \un{9} & \un{9} & \un{9} & 9 & 9 & 9 & 9 & 9 & -1 \\
21 & \un{17} & 8 & 8 & 8 & \un{9} & \un{9} & \un{9} & 9 & 9 & -2 \\
19 & \un{15} & \un{8} & 8 & 8 & 7 & 7 & 7 & \un{9} & \un{9} & -3 \\
16 & \un{12} & 5 & \un{8} & \un{8} & \un{7} & 7 & 7 & 6 & 6 & -3 \\
13 & \un{9} & 5 & 5 & 5 & 4 & \un{7} & \un{7} & \un{6} & 6 & -3 \\
10 & \un{6} & \un{5} & \un{5} & 5 & 4 & 4 & 4 & 3 & \un{6} & -2 \\
8 & \un{4} & 3 & 3 & \un{5} & \un{4} & \un{4} & 4 & 3 & 4 & -1 \\
7 & \un{3} & 3 & 3 & \un{4} & 3 & 3 & \un{4} & 3 & \un{4} & -1 \\
6 & 2 & 3 & 3 & 3 & 3 & 3 & 3 & 3 & 3
\end{array}
$$
Observe that further Cremona based on points other than one does not change the obtained ideal, so from this point only the free part will change.
Perform the same sequence on $I(0^{\times 9})_{1}$.
$$
\begin{array}{c|ccccccccc|c}
t & m_1 & m_2 & m_3 & m_4 & m_5 & m_6 & m_7 & m_8 & m_9 & k \\ \hline
1 & \un{0} & \un{0} & \un{0} & \un{0} & 0 & 0 & 0 & 0 & 0 & +2 \\
3 & \un{2} & 2 & 2 & 2 & \un{0} & \un{0} & \un{0} & 0 & 0 & +4 \\
7 & \un{6} & \un{2} & 2 & 2 & 4 & 4 & 4 & \un{0} & \un{0} & +6 \\
13 & \un{12} & 8 & \un{2} & \un{2} & \un{4} & 4 & 4 & 6 & 6 & +6 \\
19 & \un{18} & 8 & 8 & 8 & 10 & \un{4} & \un{4} & \un{6} & 6 & +6 \\
25 & \un{24} & \un{8} & \un{8} & 8 & 10 & 10 & 10 & 12 & \un{6} & +4 \\
29 & \un{28} & 12 & 12 & \un{8} & \un{10} & \un{10} & 10 & 12 & 10 & +2 \\
31 & \un{30} & 12 & 12 & \un{10} & 10 & 10 & \un{12} & 12 & \un{12} & -2 \\
29 & 28 & 12 & 12 & 8 & 10 & 10 & 10 & 12 & 10
\end{array}
$$
We use Proposition \ref{cre8} (for multiplicities $m_2,\dots,m_9$) to complete the proof.
\item
$n=17$. Consider $I((2m)^{\times 17})_{5m-1}$.
It suffices to show that $I(4m,(2m)^{\times 9})_{5m-1} = 0$.
As in the case $n=14$, only the free part will change during Cremona.
But this free part can achieve arbitrarily low degree according to Proposition \ref{3crseq}.
\item
$n=21$. Consider $I((3m)^{\times 21})_{8m-1}$.
It suffices to show that $I((5m^{\times 3},(3m)^{\times 6})_{8m-1} = 0$.
Observe that Cremona on $I(5^{\times 3},3^{\times 6})_{8}$ gives $I(3^{\times 3},1,3^{\times 5})_{6}$. Further Cremonas
on this ideal does nothing (whenever we ommit point with multiplicity one). Since the free part becomes
$I((-2)^{\times 3},-2,0^{\times 6})_{-3}$ and multiplicity $-2$ will not be used, we are done by Proposition \ref{cre8}.
\item
$n=24$. Consider $I((39m)^{\times 24})_{107m}$.
It suffices to show that $I((65m)^{\times 4},(39m)^{\times 4})_{107m} = 0$.
It is easy to observe that by a sequence of four Cremona we obtain $I(\ldots)_{-5m}$, which completes the proof.
\item
$n=30$. Consider $I(m^{\times 30})_{3m-1}$.
It suffices to show that $I((2m)^{\times 3},m^{\times 6})_{3m-1} = 0$.
After one Cremona we obtain $I((m-2)^{\times 3},m^{\times 5}_{3m-3}$ and argue similarly as in the case $n=8$.
\item
$n=35$. Consider $I((21m)^{\times 35})_{67m-1}$.
It suffices to show that 
$$I((42m)^{\times 3},35m,(21m)^{\times 6})_{67m-1} = 0.$$
Consider the following sequence of Cremona:
$$
\begin{array}{c|cccccccccc|c}
t & m_1 & m_2 & m_3 & m_4 & m_5 & m_6 & m_7 & m_8 & m_9 & m_{10} & k \\ \hline
67 & \un{42} & \un{42} & \un{42} & \un{35} & 21 & 21 & 21 & 21 & 21 & 21 & -27 \\
40 & 15 & 15 & 15 & 8 & \un{21} & \un{21} & \un{21} & \un{21} & 21 & 21 & -4 \\
36 & 15 & 15 & 15 & 8 & \un{17} & \un{17} & 17 & 17 & \un{21} & \un{21} & -4 \\
32 & 15 & 15 & 15 & 8 & 13 & 13 & \un{17} & \un{17} & \un{17} & \un{17} & -4 \\
28 & \un{15} & \un{15} & \un{15} & 8 & \un{13} & 13 & 13 & 13 & 13 & 13 & -2 \\
26 & 13 & 13 & 13 & 8 & 11 & 13 & 13 & 13 & 13 & 13
\end{array}
$$
The same sequence on $I(0^{\times 10})_{1}$ gives
$$
\begin{array}{c|cccccccccc|c}
t & m_1 & m_2 & m_3 & m_4 & m_5 & m_6 & m_7 & m_8 & m_9 & m_{10} & k \\ \hline
1 & \un{0} & \un{0} & \un{0} & \un{0} & 0 & 0 & 0 & 0 & 0 & 0 & +2 \\
3 & 2 & 2 & 2 & 2 & \un{0} & \un{0} & \un{0} & \un{0} & 0 & 0 & +6 \\
9 & 2 & 2 & 2 & 2 & \un{6} & \un{6} & 6 & 6 & \un{0} & \un{0} & +6 \\
15 & 2 & 2 & 2 & 2 & 12 & 12 & \un{6} & \un{6} & \un{6} & \un{6} & +6 \\
21 & \un{2} & \un{2} & \un{2} & 2 & \un{12} & 12 & 12 & 12 & 12 & 12 & +24 \\
45 & 26 & 26 & 26 & 2 & 36 & 12 & 12 & 12 & 12 & 12
\end{array}
$$
Now we proceed as in the case $n=21$.
\end{itemize}
\end{proof}

\begin{proposition}
\label{nicegamma}
We have the following inequalities:
\begin{itemize}
\item
$\gamma(n) \geq \gamma(k)$ for $n \geq k$;
\item
$\gamma(a \cdot 8^k) \geq 2^k \cdot \gamma(a)$, for nonnegative integers $a$ and $k$.
\end{itemize}
\end{proposition}

\begin{proof}
The first inequality is obvious. The second follows from the fact that $\gamma(8b) \geq 2\gamma(b)$ and induction.
To prove the last inequality, observe that $I(m^{\times 8})_{2m-1} = 0$ (since $\gamma(8)\geq 2$). Using Proposition \ref{glue} exactly
$b$ times, we know that if $I((2m)^{\times b})_t=0$ then $I(m^{\times 8b})_t=0$. This, together with Proposition \ref{Wc}, leads to the following
inequalities
$$\alpha(I(m^{\times 8b})) \geq \alpha(I((2m)^{\times b})) \geq 2m\gamma(b).$$
Diving both sides by $m$ and passing to the limit completes the proof. We should also observe that the equality
$\gamma(8^k) = 2^k$ follows directly from \cite{evain}.
\end{proof}

From the above proposition we easily derive the following lower bound for $\gamma(n)$:
$$\gamma(n) \geq \frac{1}{2} \sqrt[3]{n}.$$
Unfortunately, this bound is not sufficient to show containment we need. But we can find better bounds.

\begin{proposition}
\label{gamma512}
Define $$\delta(n) = \frac{3\sqrt[3]{6n}+4.3}{7}.$$
If $n \geq 512$ then 
$$\gamma(n) \geq \delta(n) \geq 0.7787 \sqrt[3]{n} + 0.6142.$$
\end{proposition}

\begin{proof}
First observe that $n \in [8^k,8^{k+1})$ for some $k \geq 3$.
We will consider several cases, depending on the position of $n$ in the interval $[8^k,8^{k+1})$.
\begin{itemize}
\item
$n \in [8^k,\frac{3}{2} \cdot 8^k]$. First, show that
\begin{equation}
\label{neq1}
1 \geq \frac{3}{7}\sqrt[3]{9}+\frac{4.3}{7\cdot 2^k}.
\end{equation}
Indeed, this is true for $k=3$ and thus for any $k \geq 3$, since the right hand side of the inequality
decreases when $k$ increases. By Proposition \ref{nicegamma}
$$\gamma(n) \geq \gamma(8^k) \geq 2^k$$
Multiplying both sides of \eqref{neq1} by $2^k$ we get exactly
$$2^k \geq \delta\left(\frac{3}{2}\cdot 8^k\right).$$
Since $\delta(n)$ is increasing with respect to $n$, we complete the proof in this case.
The rest of the proof will go exactly along the same lines.

\item
$n \in [\frac{3}{2} \cdot 8^k, 2 \cdot 8^k]$. Observe that
$$\frac{63}{57} \geq \frac{3}{7}\sqrt[3]{12}+\frac{4.3}{7\cdot 2^k}.$$
Hence
$$\gamma(n) \geq \gamma\left(\frac{3}{2} \cdot 8^k\right) = \gamma(12 \cdot 8^{k-1}) \geq 2^{k-1}\gamma(12).$$
Now we use Proposition \ref{vargamma} to bound $\gamma(12)$:
$$2^{k-1} \gamma(12) \geq 2^k \frac{63}{57} \geq \delta(2 \cdot 8^k) \geq \delta(n).$$
We will use Proposition \ref{vargamma} in all further cases.
\item
$n \in [2 \cdot 8^k,3 \cdot 8^k]$. Observe that
$$\frac{11}{9} \geq \frac{3}{7}\sqrt[3]{18}+\frac{4.3}{7\cdot 2^k}.$$
Hence
$$\gamma(n) \geq \gamma(2 \cdot 8^k) = \gamma(16 \cdot 8^{k-1}) \geq 2^{k-1}\gamma(16) \geq 2^k \frac{11}{9} \geq \delta(3 \cdot 8^k) \geq \delta(n).$$
\item
$n \in [3 \cdot 8^k, 4 \cdot 8^k]$. Observe that
$$\frac{4}{3} \geq \frac{3}{7}\sqrt[3]{24}+\frac{4.3}{7\cdot 2^k}.$$
Hence
$$\gamma(n) \geq \gamma(3 \cdot 8^k) = \gamma(24 \cdot 8^{k-1}) \geq 2^{k-1}\gamma(24) \geq 2^k \frac{4}{3} \geq \delta(4 \cdot 8^k) \geq \delta(n).$$
\item
$n \in [4 \cdot 8^k, 6 \cdot 8^k]$. Observe that
$$\frac{3}{2} \geq \frac{3}{7}\sqrt[3]{36}+\frac{4.3}{7\cdot 2^k}.$$
Hence
$$\gamma(n) \geq \gamma(4 \cdot 8^k) \geq \gamma(30 \cdot 8^{k-1}) \geq 2^{k-1}\gamma(30) \geq 2^k \frac{3}{2} \geq \delta(6 \cdot 8^k) \geq \delta(n).$$
\item
$n \in [6 \cdot 8^k, 8 \cdot 8^k]$. Observe that
$$\frac{5}{3} \geq \frac{3}{7}\sqrt[3]{48}+\frac{4.3}{7\cdot 2^k}.$$
Hence
$$\gamma(n) \geq \gamma(5 \cdot 8^k) \geq 2^k\gamma(5) \geq 2^k \frac{5}{3} \geq \delta(8 \cdot 8^k) \geq \delta(n).$$
\end{itemize}
\end{proof}

\section{Containment results}

\begin{proposition}
\label{gamma}
Let $n \geq 1$, let $I$ be an ideal of $n$ simple points in general position, let
$$\binom{s}{3} < n \leq \binom{s+1}{3},$$
let $r \geq 1$.
If either
\begin{equation}
\label{ass1}
\gamma(n) \geq \frac{(s+1)r-2}{3r-2}
\end{equation}
or
\begin{equation}
\label{ass2}
\alpha(I^{(3r-2)}) \geq (s+1)r-2
\end{equation}
then
$I^{(3r-2)} \subset M^{2r-2}I^{r}$ and $I^{(3r-1)} \subset M^{2r-1}I^{r}$ for all $r \geq 1$.
\end{proposition}

\begin{proof}
From \cite[proof of Proposition 12]{Dum} it follows that 
the Castelnuovo-Mumford regularity $\reg(I)$ of $I$ is at most $s-1$ and $I$ is generated in degrees at most $s-1$.

By \cite[Lemma 2.3.4]{BoHa1} if $r \reg(I) \leq \alpha(I^{(m)})$ then $I^{(m)} \subset I^r$. By \eqref{ass2}
$$\alpha(I^{(3r-1)}) \geq \alpha(I^{(3r-2)}) \geq (s+1)r-2 \geq (s-1)r \geq \reg(I) r$$
gives $I^{(3r-2)} \subset I^r$ and $I^{(3r-1)} \subset I^r$.

By the argument similar to used in the \cite[Lemma 4.6]{HaHu} if, for an ideal $I$ of points in $\PP^N$, $I^{(m)} \subset I^{r}$ and
$$\alpha(I^{(m)}) \geq rt+k$$
for some $t$ such that $I$ is generated in degrees $t$ and less, then
$$I^{(m)} \subset M^{k}I^r.$$

Taking $t=s-1$, using \eqref{ass2} we have
$$\alpha(I^{(3r-2)}) \geq (s+1)r-2 = (s-1)r + 2(r-1) \geq rt+2r-2,$$
which proves $I^{(3r-2)} \subset M^{2r-2}I^{r}$. Now observe that (by taking derivative)
$$\alpha(I^{(3r-1)}) > \alpha(I^{(3r-2)}),$$
and hence
$$\alpha(I^{(3r-1)}) \geq \alpha(I^{(3r-2)})+1 \geq rt+2r-1,$$
which gives $I^{(3r-1)} \subset M^{2r-1}I^{r}$.

Observe also that \eqref{ass2} follows from \eqref{ass1} by Proposition \ref{Wc}.
\end{proof}

\begin{proposition}
\label{case1}
Let $I$ be an ideal of $n \geq 512$ simple points in general position. 
Then $I^{(3r-2)} \subset M^{2r-2}I^r$ for $r \geq 3$.
\end{proposition}

\begin{proof}
Take $s$ satisfying
\begin{equation}
\label{nbetc1}
\binom{s}{3} < n \leq \binom{s+1}{3}.
\end{equation}
Define, as before,  $$\delta(n) = \frac{3\sqrt[3]{6n}+4.3}{7}.$$
By Proposition \ref{gamma512} we know that
$$\gamma(I) = \gamma(n) \geq \delta(n).$$
By a straightforward computation we can show that for $r\geq 3$
$$\delta(n) \geq \frac{(\sqrt[3]{6n}+2.1)r-2}{3r-2}.$$
From the inequality
$$s(s-1)(s-2) \geq (s-1.1)^3,$$
which holds for $s \geq 5$,
we obtain (since $6n > s(s-1)(s-2)$ by \eqref{nbetc1})
$$\sqrt[3]{6n}+1.1 \geq s$$
and hence
$$\frac{(\sqrt[3]{6n}+2.1)r-2}{3r-2} \geq \frac{(s+1)r-2}{3r-2}.$$
We conclude by Proposition \ref{gamma}.
\end{proof}

\begin{proposition}
\label{case2}
Let $I$ be an ideal of $n$ simple points in general position.
If $5 \leq n \leq 511$ then
$I^{(3r-2)} \subset M^{2r-2}I^r$ for $r \geq 3$.
\end{proposition}

\begin{proof}
Take $s$ satisfying
$$\binom{s}{3} < n \leq \binom{s+1}{3}$$
(observe that $s \leq 15$),
we want to show that
$$\gamma(n) \geq \frac{(s+1)r-2}{3r-2}$$ for
all $r \geq 3$. Since $\gamma(n)$ is increasing (with respect to $n$) and the right hand side
is decreasing (with respect to $r$) it is enough to show that
$$\gamma\left(\binom{s}{3}+1\right) \geq \frac{3s+1}{7}.$$
We will consider the following cases:
\begin{itemize}
\item
$s=14$ and $s=15$. Since $\binom{14}{3}+1 = 365 > 320$, it suffices to show that $\gamma(320) \geq \frac{3s+1}{7}$.
But 
$$\gamma(320) = \gamma(5 \cdot 8^2) \geq 4\gamma(5)$$
by Proposition \ref{nicegamma}. Now we use Proposition \ref{vargamma} to obtain
$$4\gamma(5) \geq \frac{20}{3} \geq \frac{46}{7} \geq \frac{3s+1}{7}.$$
All other cases will be treated similarly, with some changes for $s \leq 6$.
\item
$s=13$.
$$\gamma(287) \geq \gamma(240) \geq 2\gamma(30) \geq 6 \geq \frac{40}{7} = \frac{3s+1}{7}.$$
\item
$s=12$.
$$\gamma(221) \geq \gamma(168) \geq 2\gamma(21) \geq \frac{16}{3} \geq \frac{37}{7} = \frac{3s+1}{7}.$$
\item
$s=11$.
$$\gamma(166) \geq \gamma(136) \geq 2\gamma(17) \geq \frac{10}{2} \geq \frac{34}{7} = \frac{3s+1}{7}.$$
\item
$s=10$.
$$\gamma(121) \geq \gamma(112) \geq 2\gamma(14) \geq \frac{14}{3} \geq \frac{31}{7} = \frac{3s+1}{7}.$$
\item
$s=9$.
$$\gamma(84) \geq \gamma(64) \geq 4\gamma(1) \geq 4 \geq \frac{28}{7} = \frac{3s+1}{7}.$$
\item
$s=8$.
$$\gamma(57) \geq \gamma(56) \geq 2\gamma(7) \geq \frac{56}{15} \geq \frac{25}{7} = \frac{3s+1}{7}.$$
\item
$s=7$.
$$\gamma(36) \geq \gamma(35) \geq \frac{67}{21} \geq \frac{22}{7} = \frac{3s+1}{7}.$$
\item
$s=6$. Assume $n \geq 24$, then
$$\gamma(n) \geq \gamma(24) \geq \frac{107}{39} \geq \frac{19}{7} = \frac{3s+1}{7}.$$
For $n \geq 21$ we have
$$\gamma(n) \geq \gamma(21) \geq \frac{8}{3} \geq \frac{26}{10} \geq \frac{7r-2}{3r-2}$$
for $r \geq 4$. For $r=3$, by Proposition \ref{gamma} it is enough to show that \eqref{ass2} is satisfied, ie.
$$I(7^{\times 21})_{18} = 0.$$
By Proposition \ref{glue} it is enough to show that $I(14,14,7^{\times 5})_{18} = 0$.
This can be done by Cremona operations.
\item
$s=5$. Assume $n \geq 14$, then
$$\gamma(n) \geq \gamma(14) \geq \frac{7}{3} \geq \frac{16}{7} = \frac{3s+1}{7}.$$
For $n=12$ and $n=13$ we have
$$\gamma(n) \geq \gamma(12) \geq \frac{107}{39} \geq \frac{22}{10} \geq \frac{6r-2}{3r-2}$$
for $r \geq 4$. It remains to show the case $r=3$, ie. that $\alpha(I(7^{\times 12})) \geq 16$, but we will complete this together with the case $n=11$.
For $n=11$ we want to show that
$$\alpha(I(3r-2)^{\times 11})) \geq 6r-2.$$
By Proposition \ref{glue} it is enough to show that $I(4r-2,(3r-2)^{\times 7})_{6r-3} = 0$.
Consider the following Cremona operations:
$$
\begin{array}{c|ccc|c}
t & m_1 & (m_2)^{\times 3} & (m_3)^{\times 4} & k \\ \hline
6r-3 & \un{4r-2} & \un{3r-2} & 3r-2 & -r+2 \\
5r-1 & 3r & 2r & \un{3r-2} & -2r+6 \\
3r+5 & \un{3r} & \un{2r} & r+4 & -3r+10 \\
15 & 10 & 10-r & \un{r+4} & -4r+14 \\
29-4r & 10 & 10-r & 18-3r
\end{array}
$$
For $r \geq 5$ we have $29-4r < 10$ and the we are done. We are left with two cases to show:
\begin{align*}
\alpha(I(7^{\times 11})) & \geq 16, \\
\alpha(I(10^{\times 11})) & \geq 22.
\end{align*}
The above cases will be completed in Lemma \ref{singular}.
\item
$s=4$.
For $n \geq 7$ we have
$$\gamma(n) \geq \gamma(7) \geq \frac{28}{15} \geq \frac{18}{10} = \frac{(s+1)4-2}{3 \cdot 4-2}.$$
Hence for $r \geq 4$ we are done, for $r=3$ we must show that
$$\alpha(I(7^{\times n})) \geq 13.$$
This is easy with Cremona operations.
For $n=5$ and $n=6$ we will use \cite{Laf}. Consider $I((3r-2)^{\times 5})_{5r-3}$. By Cremona operation
we get
$$I(r^{\times 4},3r-2)_{3r-1}.$$
Observe that there is no Cremona operation on multiplicities $r^{\times 4}$, $3r-2$ and degree $3r-1$ which gives
lower degree.
Thus by \cite[Theorem 5.3]{Laf} there is no element in $I(r^{\times 4},3r-2)_{3r-1}$ if and only if
$$\binom{3r+2}{3} - 4\binom{r+2}{3} - \binom{3r}{3} - 1 < 0,$$
which is easy to compute.
Hence $\alpha(I((3r-2)^{\times n})) \geq 5r-2$ which completes the proof.
\end{itemize}
\end{proof}

\begin{proposition}
\label{case3}
Let $I$ be the ideal of $n \geq 65$ simple points in general position. 
Then $I^{(3r-2)} \subset M^{2r-2}I^r$ for $r = 2$.
\end{proposition}

\begin{proof}
Take $s$ satisfying
$$\binom{s}{3} < n \leq \binom{s+1}{3},$$
observe that $s \geq 5$ hence
$$\sqrt[3]{6n} + 1.1 \geq s.$$
By this inequality and Proposition \ref{gamma} it is enough to show that
$$\alpha(I(4^{\times n})) \geq 2\sqrt[3]{6n}+2.2.$$
Take the least $t$ satisfying
$$\binom{t+3}{3} \geq 20n.$$
By \cite[Theorem 1]{Bram} and \cite[Theorem 3]{uzDum} there is no element in $I(4^{\times n})_{t-1}$ if and only if
$$\binom{(t-1)+3}{3} - n\binom{6}{3} - 1 < 0.$$
Since
$$(t+3)^3 \geq (t+1)(t+2)(t+3) \geq 120n$$
we have that
$$t \geq \sqrt[3]{120n}-3 \geq 2\sqrt[3]{6n}+2.2$$
for $n \geq 65$, which completes the proof.
\end{proof}

\begin{proposition}
\label{case4}
Let $I$ be an ideal of $n$ simple points in general position.
If $7 \leq n \leq 64$ then 
$I^{(3r-2)} \subset M^{2r-2}I^r$ for $r = 2$.
\end{proposition}

\begin{proof}
Take $s$ satisfying
$$\binom{s}{3} < n \leq \binom{s+1}{3}.$$
Then $4 \leq s \leq 8$. For each $s\geq 5$, we will show that
$$I(4^{\times (\binom{s}{3}+1)})_{2s-1} = 0.$$
Observe that by \cite{Bram} and \cite{uzDum}, emptiness (ie. containing no nonzero form) of $I(4^{\times n})_{t}$ for $t \geq 9$ depends only on
its ``virtual dimension''
$$v = \binom{t+3}{3} - 20n.$$
In each case $v$ becomes negative which gives emptiness:
$$
\begin{array}{cccc}
s & 2s-1 & \binom{s}{3}+1 & v \\ \hline
8 & 15 & 57 & -325 \\
7 & 13 & 36 & -161 \\
6 & 11 & 21 & -57 \\
5 & 9 & 11 & -1
\end{array}
$$
Let $s=4$, $n \geq 7$. Use Cremona operations
$$
\begin{array}{c|cccccccccc|c}
t & m_1 & m_2 & m_3 & m_4 & m_5 & m_6 & m_7 & m_8 \ldots & k \\ \hline
7 & \un{4} & \un{4} & \un{4} & \un{4} & 4 & 4 & 4 & \ldots & -2 \\
5 & 2 & 2 & \un{2} & \un{2} & \un{4} & \un{4} & 4 & \ldots & -2 \\
3 & 2 & 2 & 0 & 0 & 2 & 2 & 4 & \ldots
\end{array}
$$
to show that $I(4^{\times n})_{7}$ is the zero ideal.
\end{proof}

\begin{lemma}
\label{singular}
Let $I$ be the ideal of $4$ or $5$ simple points in general position. Then
$$I^{(4)} \subset M^{2}I^{2}$$
Moreover, if $I$ is an ideal of $11$ simple points in general position, then
$$(I^{(7)})_{15} = (I^{(10)})_{21} = 0.$$
\end{lemma}

\begin{proof}
First of all, we may use linear automorphism of $\PP^3$ to assume that $p_1 = (1:0:0:0)$,
$p_2=(0:1:0:0)$, $p_3=(0:0:1:0)$, $p_4=(0:0:0:1)$ (see \cite{Dum}).
We take general $p_5=(a:b:c:1)$ and $p_6=(d:e:f:1)$. Now we use thory of Gr\"obner bases and Singular \cite{sing} to compute that
$I^{(4)} \subset M^{2}I^{2}$ in the ring $\field(a,b,c,d,e,f)[\PP^3]$.
{\small
\begin{verbatim}
ring R=(0,a,b,c,d,e,f),(x,y,z,w),dp; option(redSB);
ideal P1=x,y,z;
ideal P2=x,y,w;
ideal P3=x,z,w;
ideal P4=y,z,w;
ideal P5=x-a*w,y-b*w,z-c*w;
ideal P6=x-d*w,y-e*w,y-f*w;
ideal M=x,y,z,w;
ideal I5=intersect(P1,P2,P3,P4,P5);
ideal I6=intersect(I5,P6);
ideal J5=intersect(P1^4,P2^4,P3^4,P4^4,P5^4);
ideal J6=intersect(J5,P6^4);
ideal K5=quotient(M^2*I5^2,J5);
ideal K6=quotient(M^2*I6^2,J6);
K5; K6;
K5[1]=1
K6[1]=1
\end{verbatim}
}
All the computations performed by Singular will remain valid for generic values of parameters $a$, $b$, $c$, $d$, $e$, $f$ in $\field[\PP^3]$, hence
for generic $p_5$ and $p_6$ the answer would be the same, proving containment.

To show that $(I^{(7)})_{15} = (I^{(10)})_{21} = 0$ we will find the element with the lowest degree in
$I^{(7)}$ (resp. $I^{(10)}$), which can be done by computing Gr\"obner basis of the ideal with respect
to some degree ordering.  By semicontinuity of $\alpha$, it is enough to compute the above for some chosen points, moreover,
it can be performed over $\ZZ/p\ZZ$ for some prime $p$, to make computations feasible (computations over $\ZZ$ has not finished within 24 hours).
{\small
\begin{verbatim}
ring S=32000,(x,y,z,w),dp;
ideal P1=x,y,z;
ideal P2=x,y,w;
ideal P3=x,z,w;
ideal P4=y,z,w;
ideal P5=x-2w,y-w,z-w;
ideal P6=x-3w,y-2w,z-3w;
ideal P7=x-5w,y-4w,z-7w;
ideal P8=x+w,y-2w,z-6w;
ideal P9=x+12w,y-14w,z-11w;
ideal P10=x-w,y+5w,z+8w;
ideal P11=x+3w,y+w,z+8w;
ideal J7=intersect(P1^7,P2^7,P3^7,P4^7,P5^7,P6^7,P7^7,P8^7,P9^7,P10^7,P11^7);
ideal J10=intersect(P1^10,P2^10,P3^10,P4^10,P5^10,P6^10,P7^10,P8^10,P9^10,P10^10,P11^10);
J7=std(J7); J10=std(J10);
deg(J7[1]); deg(J10[1]);
16
23
\end{verbatim}
}
\end{proof}

The case of $n \leq 4$ points will be solved using \cite[Proposition 15]{Dum}:

\begin{proposition}
\label{funda}
Let $1 \leq n \leq N+1$, let $p_1=(1:0:\ldots:0)$, $p_2=(0:1:0:\ldots:0)$, $\ldots$ for $j=1,\dots,n$ be fundamental
points in $\PP^N$, let $I = I(1^{\times n})$, let $J = \cap_{j=1}^{n} \mathfrak{m}_{p_j}$. Let $r \geq 1$.
If $J^{(Nr)} \subset M^{(N-1)r}J^{r}$ then $I^{(Nr)} \subset M^{(N-1)r}I^{r}$.
\end{proposition}

For a sequence $(a_0,\dots,a_N)$ of $N+1$ nonnegative integers let
$$x^{(a_0,\dots,a_N)} = x_0^{a_0} \cdot \ldots \cdot x_N^{a_N} \in \field[\PP^3].$$
Again, by \cite[Proposition 16]{Dum} we can describe generators of $I^{(m)}$: 

\begin{proposition}
\label{gener}
Let $I$ be an ideal of $n \leq N+1$ fundamental points, let $m \geq 1$. Then $(I^{(m)})_{t}$ is generated by the following set of monomials:
$$\mathcal{M}=\mathcal{M}(n,m)_t = \bigg\{ x^{(a_0,\dots, a_N)} : \sum_{j=0}^{N} a_{j} = t, \, a_k \leq t-m \text{ for } k=0,\dots,n-1 \bigg\}.$$
\end{proposition}

\begin{proposition}
\label{case5}
Let $I$ be the ideal of at most four points in general position in $\PP^3$. Then $I^{(3r-2)} \subset M^{2r-2}I^{r}$ for all $r\geq 1$.
\end{proposition}

\begin{proof}
The proof is very similar to the proof of Proposition 18 in \cite{Dum}.
Assume $r \geq 2$ and $n \geq 2$ ($n$ denotes the number of points; for $n=1$ or $r=1$ the statement is clear).
Choose $t \geq 0$, by Proposition \ref{gener} it is sufficient to show that every element from $\mathcal{M}$ belongs to
$M^{2r-2}I^r$. Let $x^{(a_0,\dots,a_3)} \in \mathcal{M}$, of course $t \geq 3r-2$. Our aim is to show that
$$x^{(a_0,\dots,a_3)} = y \cdot z,$$
where $y$ is a product of $r$ monomials, each of them belonging to $I$,
and $z$ is a monomial of degree at least $2r-2$.
Observe that, by Proposition \ref{gener}, $x_j \in I$ for $j \geq n$, while for other indeterminates we have
$$x_jx_k \in I \text{ for } j < n, k < n.$$
Hence $y$ should be equal to the product of $r$ factors, each of them being either a single indeterminate
$x_j$ for $j \geq n$, or a product of two indeterminates $x_jx_k$ for $j,k < n$.
Let
$$s = \sum_{j=0}^{n-1} a_j, \qquad p = \sum_{j=n}^{3} a_j,$$
observe that $p+s=t$.
If $p \geq r$ then $y$ can be taken to be a product of exactly $r$ indeterminates, so $\deg y=r$.
Taking $z=x^{(a_0,\dots,a_3)}/y$ we obtain
$$\deg z = t-r \geq 3r-2-r = 2r-2,$$
hence $z \in M^{2r-2}$.
Now consider the case where $p<r$. Take $p$ single indeterminates of the form $x_j$ for $j \geq n$ and
$2(r-p)$ indetermines of the form $x_j$ for $j < n$ in such a way that their product $y$ divides $x^{(a_0,\dots,a_3)}$.
It is possible since
$$2(r-p)=2r-2p \leq t-r+2-2p = (t-p)-r+2-p \leq s.$$
Thus $y \in I^r$, let $z=x^{(a_0,\dots,a_3)}/y$, we have
$$\deg(z)=s-2(r-p).$$
Thus it suffices to show that
$$s-2(r-p) \geq 2r-2$$
which is equivalent to show that
$$t+p \geq 4r-2.$$
For $t \geq 4r-2$ this is obvious, so assume that $t < 4r-2$, ie. $t \leq 4r-3$. Since $x^{(a_0,\dots,a_3)} \in \mathcal M$ we have
$$a_j \leq t-(3r-2) \text{ for } j=0,\dots,n-1,$$
hence
$$ s \leq n(t-(3r-2)).$$
Therefore
$$p \geq t-n(t-(3r-2)),$$
hence
$$t+p \geq (2-n)t+3nr-2n.$$
But
$$(2-n)t \geq (2-n)(4r-3)$$
and thus
$$t+p \geq (2-n)(4r-3)+3nr-2n=(8-n)r+n-6 \geq 4r-2,$$
since $n=2$, $3$ or $4$.
\end{proof}

\end{document}